\documentclass[12pt]{article}
\usepackage{amsmath,amssymb}
\usepackage{amsthm}

\def\LC{\mathcal{L}}

\def\R{\mathbf{R}}

\def\1{\mathbf{1}}

\def\al{\alpha}
\def\be{\beta}
\def\pa{\partial}
\def\ep{\epsilon}
\def\de{\delta}

\def\ka{\varkappa}

\numberwithin{equation}{section}

\newtheorem{theorem}{Theorem}[section]
\newtheorem{remark}{Remark}[section]

\newcommand{\om}{\omega}
\newcommand{\Ga}{\Gamma}
\newcommand{\De}{\Delta}

\begin{document}
\title{Abstract McKean-Vlasov and HJB equations, their fractional versions and related forward-backward systems on Riemannian manifolds \thanks{The work of V.N. Kolokoltsov (Sections 1-7) was supported by the Russian Science Foundation (project No. 20-11-20119), the work of M.S. Troeva (Sections 8-10) was supported by the Ministry of Science and Higher Education of the Russian Federation (Grant No. FSRG-2020-0006).}
}
\author{
	Vassili N. Kolokoltsov\thanks{Department of Statistics, University of Warwick,
		Coventry CV4 7AL UK, and National Research University Higher School of Economics, Pokrovskii bul. 11, Moscow, Russia., Email: v.kolokoltsov@warwick.ac.uk}
	and Marianna Troeva\thanks{Research Institute of Mathematics, North-Eastern Federal University,
		58 Belinskogo str., Yakutsk 677000 Russia, Email: troeva@mail.ru
	}}		
\maketitle

\begin{abstract}
We introduce a class of abstract nonlinear fractional pseudo-differential equations
in Banach spaces that includes both the Mc-Kean-Vlasov-type equations describing nonlinear Markov processes
and the Hamilton-Jacobi-Bellman(HJB)-Isaacs equation of stochastic control and games thus allowing for
a unified analysis of these equations. This leads to an effective theory of coupled forward-backward systems
(forward McKean-Vlasov evolution and backward HJB-Isaacs evolution) that are central to the
modern theory of mean-field games. 
\end{abstract}

{\bf Mathematics Subject Classification (2010)}: {34A08, 35S15, 45G15}
\smallskip\par\noindent
{\bf Key words}: fractional McKean-Vlasov-type equations on manifolds,
fractional HJB-Isaacs equations on manifolds,
fractional forward-backward systems on manifolds, dual Banach triples,  mild solutions,
Caputo-Dzherbashyan fractional derivative, smoothing and smoothness preserving
operator semigroups

\section{Introduction}

We introduce a class of abstract nonlinear fractional pseudo-differential equations
in Banach spaces that includes both the Mc-Kean-Vlasov-type equations describing nonlinear Markov processes
and the Hamilton-Jacobi-Bellman(HJB)-Isaacs equation of stochastic control and games thus allowing for
a unified analysis of these equations. Looking at these equations as evolving in dual
Banach triples allows us to recast directly the properties of one type to the properties
of another type leading to an effective theory of coupled forward-backward systems
(forward McKean-Vlasov evolution and backward HJB-Isaacs evolution) that are central to the
modern theory of mean-field games. We are working with the mild solutions to the fractional
nonlinear equations that are based on the Zolotarev integral representation for the
Mittag-Leffler functions. The abstract setting developed allows us to include in our analysis
the related nonlinear fractional equations and forward-backward systems on manifolds yielding
results that are possibly new even for classical (not fractional) equations.
We obtain the well-posedness results for these equations. 

The present work is the continuation  of the studies by the authors in \cite{KoTr17a, KoTr18, KoTr19}.

We shall analyse the nonlinear Cauchy problems of the form
\begin{equation}
\label{eqabstrHJBMV}
\dot b(t)=Ab(t)+H(t,b(t),Db(t),\al), \quad b(a)=Y, \quad t\ge a,
\end{equation}
where $A$, $D_1, \cdots, D_n$ are unbounded linear operators in a Banach space $B$, $D=(D_1, \cdots, D_n)$, $\al$
is a parameter from another Banach space $B^{par}$ and $H$ is a continuous mapping
$\R\times B\times B^n \times B^{par}\to B$,
their fractional counterparts
\begin{equation}
\label{eqabstrHJBMVf}
D^{\be}_{a+*} b(t)=Ab(t)+H(t,b(t),Db(t), \al), \quad b(a)=Y \quad t\ge a,
\end{equation}
where $D^{\be}_{a+*}$ is the Caputo-Dzherbashyan (CD)  fractional derivative of order $\be\in (0,1)$,
\begin{equation}
\label{defCDder}
D^{\be}_{a+*} b(t)=\frac{1}{\Ga(-\be)}\int_0^{t-a}\frac{b(t-z)-b(t)}{z^{1+\be}} dz+\frac{b(t)-b(a)}{\Ga(1-\be)(t-a)^{\be}},
\end{equation}
the anticipating versions of these equations (where $H$ depends additionally on the future values of $b(s)$),
and the forward-backward systems of coupled equations of this type, which represent the main class of systems studied in the modern theory of mean-field games.

\begin{remark}
We gave the explicit formula for fractional derivative, which is a consequence of its more standard definition as
$D^{\be}_{a+*} b(t)=I_a^{1-\be} (d/dt) b(t)$ via the fractional integral $I^{\be}$.
\end{remark}

Our main examples concern the case when $B$ is a space of functions on $\R^d$
and $D$ is the gradient (derivative) operator. Specifically,
the fractional Hamilton-Jacobi-Bellmann-Isaacs (HJB-Isaacs)
equation of controlled Markov processes (with an external parameter) is the equation of the form
\begin{equation}
\label{eqHJBfr}
D^{\be}_{a+*} f(t,x)=Af(t,x)+H(t,x,f(t,x),\frac{\pa f}{\pa x}(t,x), \al),
\end{equation}
for which the most natural Banach space is $B=C_{\infty}(\R^d)$ ($C_{\infty}(\R^d)$ is the Banach space of continuous functions $f: \R^d\to \R$ tending to zero at infinity equipped with the sup-norm).
The Hamiltonian $H$ arising from optimal control usually even does not depend explicitly on $f$,
just on its gradient, and it writes down as
\begin{equation}
\label{eqHJBfrHam}
H(t,x,f,p, \al)=H(t,x,p,\al)=\sup_{u\in U} [J(t,x,u)+g(t,x,u)p],
\end{equation}
with some functions $J,g$, where $U$ is a compact set of controls
(or with $\inf\sup$ instead of just $\sup$ in case of Isaacs equations).
For such $H$ the fractional equation \eqref{eqHJBfr} was derived in
\cite{KoVe14} as a Bellman equation for optimal control of scaled limits of continuous time random walks.
The fractional version of the McKean-Vlasov type equations
 (describing nonlinear Markov processes in the sense of \cite{Ko10}) are the quasi-linear equations of type
\begin{equation}
\label{eqMVfr}
D^{\be}_{a+*} f(t,x)=A^*f(t,x)+\sum_{j=1}^d h_j(t,x,\{f(t,.)\},\al)\frac{\pa f}{\pa x_j}(t,x),
\end{equation}
for which the most natural Banach space is $L_1(\R^d)$ (or the space of Borel measures on $\R^d$).
In these equations $A$ is the generator of a Feller process in $\R^d$ and $A^*$ is its dual operators.
While $H$ in  \eqref{eqHJBfr} depends on the point-wise values of $f$, the functions $h_j$ in \eqref{eqMVfr}
usually depend on some integrals of $f$. The abstract framework of equations \eqref{eqabstrHJBMVf}
allows one to treat these cases in a unified way, as well as to include in the theory in a more or less straightforward way important new cases, for instance, fractional HJB-Isaacs or McKean-Vlasov-type equations on manifolds.

Fully analogous results of course hold for the backward versions of the Cauchy problems above,
namely for the problems
\begin{equation}
\label{eqabstrHJBMVback}
\dot b(t)=-Ab(t)+H(t,b(t),Db(t),\al), \quad b(T)=Y, \quad t\le T,
\end{equation}
and its fractional counterparts
\begin{equation}
\label{eqabstrHJBMVfback}
D^{\be}_{T-*} b(t)=- Ab(t)+H(t,b(t),Db(t), \al), \quad b(T)=Y \quad t\le T,
\end{equation}
where $D^{\be}_{T-*}$ is the right Caputo-Dzherbashyan (CD)  fractional derivative of order $\be\in (0,1)$:
\begin{equation}
\label{defCDderr}
D^{\be}_{T-*} b(t)=\frac{1}{\Ga(-\be)}\int_0^{T-t}\frac{b(t+z)-b(t)}{z^{1+\be}} dt
+\frac{b(t)-b(T)}{\Ga(1-\be)(T-t)^{\be}}.
\end{equation}

The content of the paper is as follows. In the next two sections, we recall preliminary material from the theory of the Mittag-Leffler functions and fixed point principles. 
In Section \ref{abstrMVHJB} we present results on the 
well-posedness of equations \eqref{eqabstrHJBMV} and \eqref{eqabstrHJBMVf} in the sense of mild solutions.
In Section \ref{anticipMVHJB} we prove the local well-posedness for the anticipating versions of equations \eqref{eqabstrHJBMV} and \eqref{eqabstrHJBMVf} for sufficiently small $T-a$.
In the next two sections \ref{forwbakwsystem} and \ref{fractforwbakwsystem} we present results on the well-posedness of the abstract coupled forward-backward system and the fractional version of this forward-backward system.

In Section \ref{secdualtriples} we specify the abstract model for coupled forward-backward systems and those fractional analogs looking at these equations as evolving in dual Banach triples. We obtain the local well-posedness results for the fractional coupled forward-backward system consisting of the coupled McKean Vlasov (forward) and HJB (backward) equations for $t\in [a,T]$.

In Sections \ref{seceqman} and \ref{seceqsysman}, the abstract setting developed in previous Section \ref{secdualtriples} allows us to prove the well-posedness results for the nonlinear fractional equations and fractional coupled forward-backward systems on manifolds.

We shall work everywhere with mild solutions, but the standard arguments allow one to get natural conditions that ensure that mild solutions are in fact classical, see Theorem 6.1.3 of \cite{Kobook19}. 

Note that fractional equations have become a popular subject of research due to their wide applicability in various fields of natural sciences (see e.g. \cite{leonmeersik13,Podlub99,Pskhu11}). We specially note the works \cite{AgrawGenerCalcVar10} and \cite{AtDoPiSt}, where the generalized Euler-Lagrange equations and linear fractional differential equations are considered. Numerical methods for fractional equations are presented in the papers \cite{Scalabook,Kobook21}. The problem with two-sided fractional derivatives was analyzed in \cite{HerHer16a}. Generalized fractional equations were considered in the works \cite{Kir94,KochKondr17,Ko15}.

\section{Preliminaries on the Mittag-Leffler functions}

By $E_{\be}(x)$ we denote the standard Mittag-Leffler function of index $\be$:
\[
E_{\be}(x)=\sum_{k=0}^{\infty}\frac{x^k}{\Ga(\be k +1)}.
\]
For our purpose the most convenient formula for the Mittag-Leffler function is its integral representation (Zolotarev formula, or Zolotarev-Pollard formula)
\begin{equation}
\label{eqMitLefZol}
E_{\be}(s)=\frac{1}{\be} \int_0^{\infty} e^{sx} x^{-1-1/\be} G_{\be} (1, x^{-1/\be}) \, dx,
\end{equation}
where
\[
G_{\be}(t,x)=\frac{1}{\pi} {Re} \int_0^{\infty} \exp\left\{ ipx -tp^{\be} e^{i\pi \be/2}\right\} dp
\]
is the heat kernel (solution with the Dirac initial condition) of the equation
\[
\frac{\pa G}{\pa t} (t,x)=-\frac{\pa ^{\be}}{ \pa x^{\be}} G(t,x),
\]
or, in probabilistic terms, the transition probability density of the stable L\'evy subordinator of index $\be$.
The convenience of this formula is due to the fact that it allows one to define $E_{\be}(A)$ for an operator
$A$ whenever $A$ generates a semigroup, so that $e^{At}$ is well defined.

From \eqref{eqMitLefZol} it follows that
\begin{equation}
\label{eqMitLefZol1}
E'_{\be}(s)=\frac{1}{\be} \int_0^{\infty} e^{sx} x^{-1/\be} G_{\be} (1, x^{-1/\be}) \, dx,
\end{equation}
so that the integral on r.h.s. is finite.

We also need the well known formula for the Mellin transform of $G_{\be}$:
\begin{equation}
\label{eqMitLefMel}
\int_0^{\infty} x^{-\om} G_{\be} (1, x^{-1/\be}) \, dx=\frac{\Ga(1-\om +1/\be)}{\be \Ga(\be-\be \om +1)},
\end{equation}
valid for $\om <1+1/\be$, see proof e.g. in Proposition 8.1.1 of \cite{Kobook19}.

\section{Preliminaries on the fixed-point principle for integral curves}

For a Banach space $B$ and $\tau<t$ we denote $C([\tau,t],B)$ the Banach space of continuous functions
$f:[\tau,t] \to B$ with the norm
\[
\|f\|_{C([\tau,t],B)}=\sup_{s\in[\tau,t]} \|f(s)\|_B,
\]
and $C_Y([\tau,t],B)$ its closed subset consisting of functions $f$ such that $f(\tau)=Y$, which is
a complete metric space under the induced topology.

For a closed convex subset $M$ of $B$, $C_Y([\tau,t],M)$
denotes a convex subset of $C_Y([\tau,t],B)$ of functions with values in $M$.

The following result is Theorem 2.1.3 from  \cite{Kobook19}. It is a version of the fixed-point principle,
specifically tailored to be used for nonlinear diffusions and fractional equations.

\begin{theorem}
\label{thabstractwelpos}
Suppose that for any $Y\in M$, $\al\in B_1$, with $B_1$ another Banach space,
a mapping $\Phi_{Y,\al} :C([\tau,T],M) \to C_Y([\tau,T],M)$ is given with some $T>\tau$ such that
for any $t$ the restriction of $\Phi_{Y,\al}(\mu_.)$ on $[\tau,t]$ depends only
on the restriction of the function $\mu_s$ on $[\tau,t]$. Moreover,
\begin{equation}
\label{eq1thabstractwelposfrac}
\begin{aligned}
& \|[\Phi_{Y,\al} (\mu^1_.)](t)-[\Phi_{Y,\al} (\mu^2_.)](t)\|
\le L(Y) \int_\tau^t (t-s)^{-\om}\| \mu^1_.-\mu^2_.\|_{C([\tau,s],B)} \, ds, \\
& \|[\Phi_{Y_1,\al_1} (\mu_.)](t)-[\Phi_{Y_2,\al_2} (\mu_.)](t)\| \le \ka \|Y_1-Y_2\|+\ka_1 \|\al_1-\al_2\|,
\end{aligned}
\end{equation}
for all $t\in [\tau,T]$, $\mu^1, \mu^2 \in C([\tau,T],M)$, $\al_1, \al_2 \in B_1$,
some constants $\ka, \ka_1\ge 0$, $\om \in [0,1)$,
and a continuous function $L$ on $M$.

Then for any $Y\in M$, $\al\in B_1$  the mapping $\Phi_{Y,\al}$ has a unique fixed point
$\mu_{t,\tau}(Y,\al)$ in $C_Y([\tau,T],M)$.
Moreover, if $\om>0$, then for all $t\in [\tau,T]$,
\begin{equation}
\label{eq3thabstractwelposfrac}
\|\mu_{t,\tau}(Y,\al)-Y\| \le E_{1-\om}(L(Y)\Ga(1-\om)(t-\tau)^{1-\om}) \|[\Phi_{Y,\al} (Y)](t)-Y\|,
\end{equation}
and the fixed points $\mu_{t,\tau}(Y_1,\al_1)$ and $\mu_{t,\tau}(Y_2,\al_2)$ with different initial data $Y_1,Y_2$
and parameters $\al_1,\al_2$ enjoy the estimate (for any $j=1,2$)
\begin{equation}
\label{eq4thabstractwelposfrac}
\|\mu_{t,\tau}(Y_1,\al_1)-\mu_{t,\tau}(Y_2,\al_2)\|
\le (\ka \|Y_1-Y_2\|+\ka_1\|\al_1-\al_2\|)  E_{1-\om}(L(Y_j)\Ga(1-\om)(t-\tau)^{1-\om}).
\end{equation}

If $\om=0$, these estimates are simplified to
\begin{equation}
\label{eq3thabstractwelpos}
\|\mu_{t,\tau}(Y,\al)-Y\| \le e^{(t-\tau)L(Y)} \|[\Phi_{Y,\al} (Y)](t)-Y\|,
\end{equation}
 \begin{equation}
\label{eq4thabstractwelpos}
\|\mu_{t,\tau}(Y_1,\al_1)-\mu_{t,\tau}(Y_2,\al_2)\|
\le (\ka \|Y_1-Y_2\|+ \ka_1 \|\al_1-\al_2\|) \exp\{(t-\tau)\min(L(Y_1), L(Y_2)) \}.
\end{equation}
\end{theorem}

\section{Abstract fractional McKean-Vlasov and HJB equations}
\label{abstrMVHJB}

For two Banach spaces $B,C$ we denote by $\LC(B,C)$ the Banach space of bounded linear operators
$B\to C$ with the usual operator norm denoted $\|.\|_{B\to C}$.

The sequences of embedded Banach spaces $B_2\subset B_1\subset B$
with the norms denoted $\|.\|_2, \|.\|_1, \|.\|$ respectively, will be
referred to as the {\it Banach triple} (of embedded spaces) or a {\it Banach tower of order} $3$, if the
norms are ordered, $\|.\|_2\ge  \|.\|_1\ge  \|.\|$, and
$B_2$ is dense in $B_1$ in the topology of $B_1$ while $B_1$ is dense in $B$ in the topology of $B$.
The following setting will play the key role in this paper.

Conditions (A):

(i)  Let $B_2\subset B_1\subset B$ be the Banach triple,
with the norms denoted $\|.\|_2, \|.\|_1, \|.\|$ respectively, and let
\[
D_i\in \LC(B_1,B)\cap \LC(B_2,B_1), \quad i=1,\cdots, n.
\]
Without loss of generality we assume that norms of all $D_j$ are bounded by $1$
in both $\LC(B_1,B)$ and $\LC(B_2,B_1)$ (which is usually the case in applications below).

(ii) Let $A\in \LC(B_2,B)$ and let $A$ generate a strongly continuous
semigroup $e^{At}$ in both $B$ and $B_1$, so that
\begin{equation}
\label{AgeninBB1}
\|e^{At}\|_{B\to B} \le Me^{mt}, \quad \|e^{At}\|_{B_1\to B_1} \le M_1 e^{t m_1},
\end{equation}
with some nonnegative constants $M,m, M_1,m_1$, and $B_2$ is an invariant core for this semigroup in $B$.

(iii) Let $B^{par}$ be another Banach space (of parameters) with the norm denoted $\|.\|_{par}$
and $H:\R\times B\times B^n\times B^{par}\to B$ be a continuous mapping, which is Lipschitz in the sense that
\[
\|H(t,b_0,b_1, \cdots, b_n,\al)-H(t,\tilde b_0, \tilde b_1,\cdots, \tilde b_n,\al)\|
\]
\begin{equation}
\label{HamLipsch}
\le L_H \sum_{j=0}^n \|b_j-\tilde b_j\|(1+L'_H\sum_{j=1}^n \|b_j\|),
\end{equation}
\begin{equation}
\label{HamLipschpar}
\|H(t,b_0,b_1, \cdots, b_n,\al)-H(t,b_0, b_1,\cdots, b_n, \tilde \al)\|
\le L_H^{par}\|\al-\tilde \al\|_{par} (1+\sum_{j=1}^n \|b_j\|),
\end{equation}
and is of linear growth
\begin{equation}
\label{HamLingrC}
\|H(t,b_0,b_1, \cdots, b_n,\al)\| \le L_H \sum_{j=1}^n \|b_j\|,
\end{equation}
with some constants $L_H, L'_H, L_H^{par}$.

\begin{remark} For the classical HJB equations estimate \eqref{HamLipsch}
 holds with $L'_H=0$, in which case \eqref{HamLingrC} follows from  \eqref{HamLipsch} and \eqref{HamLipschpar}
 (with $L_H$ that may depend on $\al$).
 However, for McKean-Vlasov-type equations, the linear growth
 of the Lipschitz constant in \eqref{HamLipsch} is not avoidable.
\end{remark}

An important assumption in our analysis is the following smoothing
property of the semigroup $e^{At}$: for $t>0$ it takes $B$ to $B_1$ and
\begin{equation}
\label{eqsmoothimngas}
\|e^{At}\|_{B\to B_1}\le \ka t^{-\om}, \quad t\in (0,1],
\end{equation}
with some constants $\ka>0$ and $\om \in (0,1)$.
Sometimes a similar condition for the pair $B_1,B_2$ is used:
\begin{equation}
\label{eqsmoothimngas2}
\|e^{At}\|_{B_1\to B_2}\le \ka_1 t^{-\om}, \quad t\in (0,1].
\end{equation}

\begin{remark}
For pseudo-differential operators $A$ (including the generators of Feller semigroups)
the deeper smoothing property \eqref{eqsmoothimngas2} can be derived
from \eqref{eqsmoothimngas} and the smoothness of the symbol of $A$, see Theorem 5.15.1 in \cite{Kobook19}.
\end{remark}

The so-called mild version of the Cauchy problem \eqref{eqabstrHJBMV} is the integral equation
\begin{equation}
\label{eqabstrHJBMVmil}
b(t)=e^{A(t-a)}Y+\int_a^t e^{A(t-s)} H(s,b(s),Db(s),\al) \, ds, \quad t\ge a.
\end{equation}
It is well known (see e.g. \cite{Kobook19}) and easy to see that if $b(t)$ solves equation \eqref{eqabstrHJBMV},
then it solves also equation \eqref{eqabstrHJBMVmil}, so that the uniqueness for \eqref{eqabstrHJBMVmil}
implies the uniqueness for \eqref{eqabstrHJBMV}.

The following theorem on the well-posedness of the equations \eqref{eqabstrHJBMVmil} is valid.
\begin{theorem}
\label{thabsrtHJBMVwell}
Let conditions (A)
and smoothing property \eqref{eqsmoothimngas} hold. Then
equation \eqref{eqabstrHJBMVmil} is well posed in $B_1$, that is, for any
$Y\in B_1$, $\al \in B^{par}$ there exists its unique global solution $b(t)=b(t;Y,\al)\in B_1$,
which depends Lipschitz continuously on the initial data $Y$ and the parameter $\al$.
In particular,
\begin{equation}
\label{eq00thabsrtHJBMVwell}
\sup_{t\in[a,T]}\|b(t;Y,\al)-b(t;\tilde Y,\tilde \al)\|_1
\le K\left(\|\al-\tilde \al\|_{par}(1+\|Y\|_1)+\|Y-\tilde Y\|_1\right),
\end{equation}
with constant $K$ depending on $t-a$ and all constants entering the assumptions of the theorem.
\end{theorem}

\begin{proof}
Solutions to \eqref{eqabstrHJBMVmil} are fixed points of the  mapping
\begin{equation}
\label{eq0thabsrtHJBMVwell}
[\Phi_{Y,\al}(b(.))](t)=e^{A(t-a)}Y+\int_a^t e^{A(t-s)} H(s,b(s),Db(s),\al) \, ds
\end{equation}
acting in $C_Y([a,T],B_1)$ for any $T>a$. The fact that it takes this space to itself follows directly
from the assumptions of the Theorem.

Assume first that $L'_H=0$ in \eqref{HamLipsch}.
Then it follows that
\begin{equation}
\label{eq1thabsrtHJBMVwell}
\|[\Phi_{Y_1,\al}(b(.))](t) -[\Phi_{Y_2,\al}(b(.))](t)\|_1
\le  M_1e^{m_1t} \|Y_1-Y_2\|_1,
\end{equation}
and
\[
\|[\Phi_{Y,\al}(b^1(.))](t) -[\Phi_{Y,\al}(b^2(.))](t)\|_1
\]
\[
\le  \ka L_H (T-a)^{\om}M_1 e^{m_1(T-a)}\int_a^t (t-s)^{-\om} \left(\|b^1(s)-b^2(s)\|
+\sum_{j=1}^n \|D_jb^1(s)-D_jb^2(s)\|\right) \, ds
\]
\begin{equation}
\label{eq2thabsrtHJBMVwell}
\le  \ka L_H (n+1)(T-a)^{\om}M_1 e^{m_1(T-a)}\int_a^t (t-s)^{-\om} \|b^1(s)-b^2(s)\|_1 \, ds.
\end{equation}
Note that the coefficient $(T-a)^{\om}M_1 e^{m_1(T-a)}$ appears here because
smoothing \eqref{eqsmoothimngas} is assumed only for $t<1$.
Due to these inequalities, the well-posedness for any $\al$ follows
from Theorem \ref{thabstractwelpos}. Moreover, from \eqref{HamLipschpar}, we find that
 \[
\|[\Phi_{Y,\al}(b(.))](t) -[\Phi_{Y,\tilde \al}(b(.))](t)\|_1
\le \tilde K\|\al-\tilde \al\|_{par} (1+\|b(.)\|_{C([\tau,T],B_1)}),
\]
and therefore, if $b$ is a fixed point of $\Phi$, its growth is bounded
by \eqref{eq3thabstractwelpos},  and thus
 \[
\|[\Phi_{Y,\al}(b(.))](t) -[\Phi_{Y,\tilde \al}(b(.))](t)\|_1
\le K\|\al-\tilde \al\|_{par} (1+\|Y\|_1),
\]
with some constants $\tilde K, K$ depending on the constants entering the assumptions of the theorem.
Hence Theorem \ref{thabstractwelpos} (with $\ka_1$ from \eqref{eq1thabstractwelposfrac} depending on $\|Y\|_1$)
again applies to give \eqref{eq00thabsrtHJBMVwell}.

If $L'_H\neq 0$ in  \eqref{HamLipsch}, an additional preliminary step is needed in the proof.
Namely, one has to show that all iterations are uniformly bounded in $B_1$. This follows from the
assumptions of linear growth \eqref{HamLingrC} and then applies the previous argument. 
\end{proof}

\begin{remark} 
Of course,  Theorem \ref{thabstractwelpos} supplies also explicit estimates
for the constant $K$ and the growth of solutions in time.
\end{remark}

To treat equation \eqref{eqabstrHJBMVf} we recall that its mild form is the integral equation
\begin{equation}
\label{eqabstrHJBMVmilfr}
b(t)=E_{\be}(A(t-a)^{\be})Y+\be \int_a^t (t-s)^{\be-1} E_{\be}'(A(t-s)^{\be}) H(s,b(s),Db(s),\al) \, ds,
\end{equation}
where $E_{\be}(A)$ is defined by \eqref{eqMitLefZol}. 
Thus more explicitly this equation writes down as
\[
b(t)=\frac{1}{\be} \int_0^{\infty} e^{A(t-a)^{\be} x} Y x^{-1-1/\be} G_{\be} (1, x^{-1/\be}) \, dx
\]
\begin{equation}
\label{eqabstrHJBMVmilfr1}
+\int_a^t (t-s)^{\be-1}  \int_0^{\infty} e^{A(t-s)^{\be} x} x^{-1/\be} G_{\be} (1, x^{-1/\be}) \, dx  H(s,b(s),Db(s),\al) \, ds.
\end{equation}
Again one proves (see Theorem 8.2.1 of \cite{Kobook19}) that
any solution of problem \eqref{eqabstrHJBMVf} solves \eqref{eqabstrHJBMVmilfr}.

\begin{theorem}
\label{thabsrtHJBMVwellfr}
Let conditions (A)
and smoothing \eqref{eqsmoothimngas} hold. Then
equation \eqref{eqabstrHJBMVmilfr} is well posed in $B_1$, that is, for any
$Y\in B_1$, $\al \in B^{par}$ there exists its unique global solution $b(t)=b(t;Y,\al)\in B_1$,
which depends Lipschitz continuously on the initial data $Y$ and parameter $\al$
so that \eqref{eq00thabsrtHJBMVwell} holds.
\end{theorem}

\begin{proof}
For simplicity, let us discuss again only the case with $L'_H= 0$ in  \eqref{HamLipsch}.
Solutions to \eqref{eqabstrHJBMVmil} are fixed points of the  mapping	
\[
[\Phi_{Y,\al}(b(.))](t)=E_{\be}(A(t-a)^{\be})Y+\be \int_a^t (t-s)^{\be-1} E_{\be}'(A(t-s)^{\be}) H(s,b(s),Db(s),\al) \, ds
\]
acting in $C_Y([a,T],B_1)$ for any $T>a$.
We have
\begin{equation}
\label{eq1thabsrtHJBMVwellfr}
\|[\Phi_{Y_1,\al}(b(.))](t) -[\Phi_{Y_2,\al}(b(.))](t)\|_1
\le  M_1 E_{\be}(m_1(t-a)^{\be}) \|Y_1-Y_2\|_1,
\end{equation}
which follows from \eqref{AgeninBB1} and \eqref{eqMitLefZol}.
Next,
\[
[\Phi_{Y,\al}(b^1(.))](t) -[\Phi_{Y,\al}(b^2(.))](t)
=\int_a^t (t-s)^{\be-1} ds
\]
\[
\times \int_0^{\infty} e^{A(t-s)^{\be} x} x^{-1/\be} G_{\be} (1, x^{-1/\be})dx
[H(s,b^1(s),Db^1(s),\al)-H(s,b^2(s),Db^2(s),\al)] \, ds.
\]
Decomposing the double integral in two parts over the sets $(t-a)^{\be} x\le 1$
and its complement, we use estimate \eqref{eqsmoothimngas} in the first integral and the second estimate from \eqref{AgeninBB1}
in the second integral yielding
\[
\|[\Phi_{Y,\al}(b^1(.))](t) -[\Phi_{Y,\al}(b^2(.))](t)\|_1
\]
\[
\le  \ka L_H (n+1) \int_a^t (t-s)^{\be-1} (t-s)^{-\om \be} \int_0^{\infty} x^{-\om -1/\be} G_{\be} (1, x^{-1/\be}) \, dx
\|b^1(s)-b^2(s)\|_1 \, ds
\]
\[
+ \ka L_H M_1(n+1) \int_a^t (t-s)^{\be-1} \int_0^{\infty} e^{m_1(t-s)^{\be} x}  x^{-1/\be} G_{\be} (1, x^{-1/\be}) \, dx
\|b^1(s)-b^2(s)\|_1 \, ds.
 \]
Using \eqref{eqMitLefMel} to estimate the first term and  \eqref{eqMitLefZol1}
 to estimate the second term  yields
\[
\|[\Phi_{Y,\al}(b^1(.))](t) -[\Phi_{Y,\al}(b^2(.))](t)\|_1
\]
\begin{equation}
\label{eq2thabsrtHJBMVwellfr}
\le   L_H (n+1)(\ka+M_1) C(m_1,\be,T-a) \int_a^t (t-s)^{-[1-\be(1-\om)]} \|b^1(s)-b^2(s)\|_1 \, ds,
\end{equation}
with a constant $C$ depending on $m_1,\be, T-a$.
Due to these inequalities, and similar inequality expressing the Lipschitz dependence of $\Phi$ on $\al$,
 the claim follows again from Theorem \ref{thabstractwelpos}.
\end{proof}

\begin{remark} Notice that $E_1(x)=e^x$ and thus $E'_1(x)=e^x$. Hence the fractional mild equation \eqref{eqabstrHJBMVmilfr}
turns to the classical mild equation \eqref{eqabstrHJBMVmil} as $\be\to 1$. Hence the discussion of the case of classical
differential equations can be considered as included in the fractional theory. In later sections, we shall sometimes talk
about just fractional equations having in mind that the classical case is recovered automatically by setting $\be=1$.
\end{remark}

\begin{remark}
\label{remonback}
All the results above have their straightforward counterparts for the backward problems
under the same exact conditions. The only difference is that, instead of the sets $C_Y([\tau,t],B)$
of functions with a fixed left point, in the backward setting one is working with the sets
$C^b_Z([\tau,T],B)$ of functions with a fixed right point: $f(T)=Z$.
\end{remark}

Let us comment on how the results above are applied to HJB and McKean-Vlasov equations
\eqref{eqHJBfr} and \eqref{eqMVfr}.
For these cases, $B$ is a space of functions on $\R^d$,
$A$ is the generator of a Feller process in $R^d$
and $D$ is the gradient (derivative) operator.
Specifically, for HJB equation the natural Banach triple is
$C^2_{\infty}(\R^d)\subset C^1_{\infty}(\R^d) \subset C_{\infty}(\R^d)$,
and for the McKean-Vlasov equations a possible triple is
$W_2(\R^d)\subset W_1(\R^d)\subset L_1(\R^d)$.
Here $C_{\infty}(\R^d)$ denotes the space of continuous functions on $\R^d$ tending to zero at infinity,
$C^j_{\infty}(\R^d)$ their subsets of functions having derivatives of order up to $j$ in $C_{\infty}(\R^d)$,
and $W_j(R^d)$ denote the Sobolev spaces of functions with partial derivatives (understood in the sense of
generalized functions) of order up to $j$ in $L_1(\R^d)$,
equipped with the integral (Sobolev) norms
\[
\|f\|_{L_1(R^d)}=\int |f(x)| \, dx, \quad
\|f\|_{W_1(R^d)}=\|f\|_{L_1(R^d)}+\sum_{j=1}^d \int \left|\frac{\pa f}{\pa x_j}\right| dx,
\]
\[
\|f\|_{W_2(R^d)}=\|f\|_{W_1(R^d)}+\sum_{i\le j} \int \left|\frac{\pa^2 f}{\pa x_j\pa x_i}\right|dx.
\]

The proof of the smoothing properties of $A$ is usually based on the properties of the Green functions
of operator $A$ (transition probabilities of the processes generated by $A$). For instance,
\eqref{eqsmoothimngas} and \eqref{eqsmoothimngas2} are known to hold with $\om=1/2$ for non-degenerate
diffusion operators $A$ with smooth enough coefficients.
When $A=-|\De|^{\al}$ with $\al\in (1,2)$, or more  generally $A=-a(x)|\De|^{\al}$ with a smooth
positive function $a$ bounded from above and below, or even more generally when $A$ is a pseudo-differential
operator that generates a non-degenerate
stable-like process with a symmetric spectral measure, that is
\[
A=\int_{S^{d-1}} |(\nabla, s)|^{\al} \mu(s) \, ds,
\]
with a positive smooth function $\mu$ on $(d-1)$-dimensional sphere $S^{d-1}$, bounded from above and below,
the semigroup generated by $A$ satisfies \eqref{eqsmoothimngas}, \eqref{eqsmoothimngas2} and \eqref{AgeninBB1},
with $\om=1/\al$, as shown in \cite{Ko09} (see also books \cite{Ko11} and \cite{Kobook19}).
 Similar properties are also being held for various mixtures of diffusions and stable processes perturbed
 by pure jump processes. As shown in Theorem \ref{thAindualtriples} below,
  from the properties of $A$ one can derive analogous properties for $A^*$
by duality arguments. Thus for these $A$ the required assumptions on $A$ and $A^*$ from
\eqref{eqHJBfr} and \eqref{eqMVfr} hold. We also refer to Section
5.15 of book \cite{Kobook19}, where it is shown, how one can derive \eqref{eqsmoothimngas2} from
a more simple estimate \eqref{eqsmoothimngas} and additional smoothness. For Hamiltonian functions
\eqref{eqHJBfrHam} assumptions (B) are often difficult to check, and therefore the assumptions (C) were introduced.
The theory of fractional HJB equations was initially developed in \cite{KoVe14a}, with more detail given
in  \cite{Kobook19}.

\begin{remark}
From an abstract point of view the spaces $C_{\infty}(\R^d)$ and $L_1(\R^d)$ are the basic examples of
abstract $AL$ and $AM$ spaces (see \cite{SchLatt}).
\end{remark}

\section{Anticipating equations}
\label{anticipMVHJB}

Anticipating versions of equations \eqref{eqabstrHJBMV} and \eqref{eqabstrHJBMVf} can be stated as the problems
\begin{equation}
\label{eqabstrHJBMVan}
\dot b(t)=Ab(t)+H(t,b(t),Db(t),u(b(.))), \quad b(a)=Y, \quad t\in [a,T],
\end{equation}
\begin{equation}
\label{eqabstrHJBMVfan}
D^{\be}_{a+*} b(t)=Ab(t)+H(t,b(t),Db(t), u(t, b(.))), \quad b(a)=Y \quad t\ge a,
\end{equation}
where $A$, $D_1, \cdots, D_n$, $H$ are as in \eqref{eqabstrHJBMV} and \eqref{eqabstrHJBMVf},
and $u$ is a continuous mapping $\R \times C([a,T],B_1) \to B^{par}$, which is Lipschitz in the second argument, so that
\begin{equation}
\label{eqabstrHJBMVancond}
\|u(t, b(.))-u(t, \tilde b(.))\|_{par}\le L_u\|b(.)-\tilde b(.)\|_{C([a,T],B_1)}.
\end{equation}
Usually, in applications $u(t,.)$ depends actually only on the future $b(s): s\in [t,T]$,
but this additional condition does not simplify analysis.

For these problems uniqueness usually does not hold globally (that is for large $T$), only
existence can be proved under rather general assumptions  (see e.g. \cite{KolWei} for equations of type \eqref{eqabstrHJBMVan}).
We shall prove here the local well-posedness, that is, for sufficiently small $T-a$.
Of  course, we again work with mild solutions, relying on the fact that any solution of  \eqref{eqabstrHJBMVan}
is a fixed point of the mapping
\begin{equation}
\label{eqfixmapan}
[\Phi_Y(b(.))](t)=e^{A(t-a)}Y+\int_a^t e^{A(t-s)} H(s,b(s),Db(s),u(s,b(.))) \, ds,
\end{equation}
and any solution of  \eqref{eqabstrHJBMVfan} is a fixed point of the mapping
\begin{equation}
\label{eqfixmapanfr}
[\Phi_Y(b(.))](t)=E_{\be}(A(t-a)^{\be})Y+\be \int_a^t (t-s)^{\be-1} E_{\be}'(A(t-s)^{\be}) H(s,b(s),Db(s),u(s,b(.))) \, ds.
\end{equation}

Due to the anticipating dependence of $u$ on $b$, Theorem \ref{thabstractwelpos} is not applicable.
We shall use just the standard Banach contraction principle. 

\begin{theorem}
\label{thabsrtHJBMVan}
Let conditions (A),
smoothing property \eqref{eqsmoothimngas} and Lipschitz estimate \eqref{eqabstrHJBMVancond} hold.
Then, for any $r>0$, there exists $T_0$ such that equations
\eqref{eqabstrHJBMVan} and \eqref{eqabstrHJBMVfan} have unique mild solutions $b(t)=b(t;Y)$
(that is, they are fixed points of \eqref{eqfixmapan} and \eqref{eqfixmapanfr} respectively)
for all $T<T_0$ and $Y$ such that $\|Y\|\le r$, and this solution depends Lipschitz continuously
 on $Y$ inside the ball $\|Y\|\le r$.
\end{theorem}

\begin{proof} Due to formula \eqref{eqMitLefZol} expressing Mittag-Leffler functions in terms of exponents,
the proof is essentially the same for both equations. So let us discuss only equation \eqref{eqabstrHJBMVan}.
There are two steps. Firstly, due to the linear growth condition \eqref{HamLingrC}, one derives for
$\|Y\|_1\le r$  the estimate
\[
\|[\Phi_Y(b(.))](t)\|_1 \le M_1e^{m_1T}r + \ep \|b(.)\|_{C_Y([a,T],B_1)},
\]
where $\ep$ depends on $T_0$, so that one can choose $T_0$ in such a way that $\ep <1$.
In this case if $\sup_t\|b(.)\|_1\le R$, then
\[
\|[\Phi_Y(b(.))](t)\|_1 \le  M_1e^{m_1T}r+\ep R \le R,
\]
whenever
\[
R>\frac{M_1e^{m_1T}r}{1-\ep}.
\]
Thus for any such $R$, $\Phi_Y$ takes $C_Y([a,T],B_1^R)$
to itself, where $B_1^R$ is the ball in $B_1$ of radius $R$ centred at $0$.
Next, from the estimate
\[
\|[\Phi_Y(b^1(.))](t)-\Phi_Y(b^2(.))(t)\|_1
\le \|b^1(.)-b^2(.)\|_{C_Y([a,T],B_1)} (T-a) e^{m_1(T-a)} [k_1+k_2 (\|Y\|_1+R)],
\]
where $k_1, k_2$ do not depend on $R$ and $Y$,
it follows that $\Phi_Y$ is a contraction in the complete metric spaces
$C_Y([a,T],B_1^R)$ for sufficiently small $T-a$.
\end{proof}

\section{Forward-backward systems}
\label{forwbakwsystem}

The well-posedness result of this section is an abstract version of the results obtained in
\cite{KolWei} (slightly extended in Chapter 6 of book \cite{Kobook19}). This abstract presentation
not only simplifies the exposition, but it is specifically designed for a more or less straightforward extension
to the fractional case, as given in the next section, and which is our main concern here.

Paper \cite{KolWei} contains also a global existence result (without uniqueness) that can be recast
in the present abstract setting, but we are not giving detail here.

To introduce our coupled forward backward system we need to introduce conditions (A) both
for the forward and the backward parts of the system plus the coupling mechanism (interpreted as control in applications)
including an appropriate setting of the parameter space to support this mechanism.
That is, the forward-backward version of conditions (A) reads as follows.

Conditions (AFB):

(i) Let $B_2\subset B_1\subset B$ and $B_{2b}\subset B_{1b}\subset B_b$ be two Banach triples
(index 'b' arising from 'backward'), with the norms in the  second triple denoted
$\|.\|_{2b}$, $\|.\|_{1b}$ and $\|.\|_b$.
Let $D=(D_1, \cdots, D_n)$ and $D^b=(D_{1b}, \cdots, D_{nb})$ be operators satisfying conditions (A) (i)
for the triples $B_2\subset B_1\subset B$ and $B_{2b}\subset B_{1b}\subset B_b$, respectively.

(ii) Let $A$ and $A^b$ be operators with properties described in condition (A)(ii) for
the triples $B_2\subset B_1\subset B$ and $B_{2b}\subset B_{1b}\subset B_b$, respectively.
The corresponding constants governing the growth of semigroups for the second triple will be denoted
 $M_b,m_b, M_{1b},m_{1b}$.

(iii) Let $B^{con}$ be another Banach space (of control functions) with the norm
denoted $\|.\|_{con}$. Let $H:\R\times B\times B^n\times B^{con}\to B$ be a continuous
mapping, which is Lipschitz in the sense that
\begin{equation}
\label{HamLipschrep}
\|H(t,b_0,b_1, \cdots, b_n,u)-H(t,\tilde b_0, \tilde b_1,\cdots, \tilde b_n,u)\|
\le L_H \sum_{j=0}^n \|b_j-\tilde b_j\|(1+L'_H\sum_{j=1}^n \|b_j\|),
\end{equation}
\begin{equation}
\label{HamLipschparrep}
\|H(t,b_0,b_1, \cdots, b_n,u)-H(t,b_0, b_1,\cdots, b_n, \tilde u)\|
\le L_H^{par}\|u-\tilde u\|_{con} (1+\sum_{j=1}^n \|b_j\|).
\end{equation}
Let $H^b:\R\times B_b\times (B_b)^n\times C([a,T],B)\to B_b$ be a continuous mapping,
which is Lipschitz in the sense that
\begin{equation}
\label{HamLipschrepb}
\|H^b(t,f_0,f_1, \cdots, f_n,b(.))-H^b(t,\tilde f_0, \tilde f_1,\cdots, \tilde f_n,b(.))\|_b
\le L^b_H \sum_{j=0}^n \|f_j-\tilde f_j\|_b (1+(L_H^b)'\sum_{j=1}^n \|b_j\|_b),
\end{equation}
\begin{equation}
\label{HamLipschparrepb}
\|H^b(t,f_0,f_1, \cdots, f_n,b(.))-H^b(t,f_0, f_1,\cdots, f_n, \tilde b(.))\|
\le L_H^b\|b-\tilde b\|_{C([a,T],B)} (1+\sum_{j=0}^n \|f_j\|_b).
\end{equation}
Finally let both $H$ and $H^b$ be of linear growth (in the sense of equation \eqref{HamLingrC}).

(iv) Let $u: \R\times B_b \times (B_b)^n \to B^{con}$ be a continuous function,
which is Lipschitz continuous in the sense that
\begin{equation}
\label{eqcontLipcon}
\|u(t,f_0,f_1, \cdots, f_n)-u(t,\tilde f_0, \tilde f_1,\cdots, \tilde f_n)\|
\le L_u \sum_{j=0}^n \|f_j-\tilde f_j\|_b,
\end{equation}
with a constant $L_u$.

The forward-backward system we are analysing here is of the form
\begin{equation}
\label{eqabstrforback}
\left\{
\begin{aligned}
& \dot b(t)=Ab(t)+H(t,b(t),Db(t),u(t,f(t),Df(t))), \quad b(a)=Y, \quad t\in [a,T], \\
& \dot f(t)=-A^bf(t)+H^b(t,f(t),D^bf(t), b_{\ge t}), \quad f(T)=Z, \quad t\in [a,T].
\end{aligned}
\right.
\end{equation}

The notation $b_{\ge t}$ indicates the assumption that $H(t,.)$ depends only on the future
 values $\{b(s), s\in [t,T]\}$ of curves $b(.)\in C([a,T],B)$.

Our approach for the analysis of system \eqref{eqabstrforback} is based on its reduction to
a single anticipating equation of type \eqref{eqabstrHJBMVan}.
Namely, by Theorem \ref{thabsrtHJBMVwell} (and Remark \ref{remonback})
for a given  curve $b(.)\in C_Y([a,T],B)$ and $Z\in B_{1b}$ there exists a unique
mild solution $f(t)=f(t; Z, b(.))\in C([a,T],B_{1b})$ of the second equation of system  \eqref{eqabstrforback}.
Substituting this solution into the first equation of  system  \eqref{eqabstrforback} we get the equation
\begin{equation}
\label{eqabstrforbackred}
\dot b(t)=Ab(t)+H(t,b(t),Db(t),u(t,f(t;Z,b(.)),Df(t;Z,b(.)))), \quad t\in [a,T].
\end{equation}
This equation is of type \eqref{eqabstrHJBMVan}.

\begin{theorem}
\label{thabsrforback}
Let conditions (AFB)
and smoothing property \eqref{eqsmoothimngas} for $A$ and $A^b$ hold. Then, for any
$Y\in B_1$, $Z\in B_{1b}$ there exists $T_0$ such that the forward-backward system
\eqref{eqabstrforback} has unique mild solution $(b(t),f(t))$
for all $T<T_0$, which depends locally Lipschitz on $Y$, $Z$, that is, when are $Y$ and $Z$
are taken from bounded sets.
\end{theorem}

\begin{proof} The discussion above shows that the statement is reduced to the well-posedness
of equation \eqref{eqabstrforbackred}.
To apply Theorem \ref{thabsrtHJBMVan} we just have to check
the Lipschitz estimate \eqref{eqabstrHJBMVancond}, which in the present case has the form
\begin{equation}
\label{eq1thabsrforback}
\|u(t,f(t;Z,b(.)),Df(t;Z,b(.)))-u(t,f(t;Z,\tilde b(.)),Df(t;Z,\tilde b(.)))\|_{con}
\end{equation}
\[
\le \|b(.)-\tilde b(.)\|_{C([a,T],B_1)}.
\]
%\end{equation}
But this holds, because the function $u$ is a Lipschitz continuous function of its arguments
(by \eqref{eqcontLipcon}) and
$f$ and $Df$ are Lipschitz functions of $b(.)$ by Theorem \ref{thabsrtHJBMVwell}.
Let us stress that this application was the main reason for us to get the continuous
dependence of the solutions of the HJB and McKean-Vlasov equations in a 'deeper' space
$B_1$ (not just in $B$).
\end{proof}

\section{Fractional forward-backward systems}
\label{fractforwbakwsystem}

The fractional analogue of system   \eqref{eqabstrforback} is ,of course, the system
\begin{equation}
\label{eqabstrforbackfr}
\left\{
\begin{aligned}
& D_{a+*} b(t)=Ab(t)+H(t,b(t),Db(t),u(t,f(t),Df(t))), \quad b(a)=Y, \quad t\in [a,T], \\
& D_{T-*} f(t)=-A^bf(t)+H^b(t,f(t),D^bf(t), b_{\ge t}), \quad f(T)=Z, \quad t\in [a,T].
\end{aligned}
\right.
\end{equation}

Under the same assumptions (AFB) we can conclude from Theorem \ref{thabsrtHJBMVwellfr} that
for a given  curve $b(.)\in C_Y([a,T],B)$ and $Z\in B_{1b}$ there exists a unique
mild solution $f(t)=f(t; Z, b(.))\in C([a,T],B_{1b})$ of the second equation of system  \eqref{eqabstrforbackfr}.
Substituting this solution into the first equation of  system  \eqref{eqabstrforbackfr} we get the equation
\begin{equation}
\label{eqabstrforbackredfr}
D_{a+*} b(t)=Ab(t)+H(t,b(t),Db(t),u(t,f(t;Z,b(.)),Df(t;Z,b(.)))), \quad t\in [a,T],
\end{equation}
which is of type \eqref{eqabstrHJBMVfan}.

The following result is proved by the same exact argument as Theorem \ref{thabsrforback}:

\begin{theorem}
\label{thabsrforbackfr}
Let conditions (AFB)
and smoothing property \eqref{eqsmoothimngas} for $A$ and $A^b$ hold. Then, for any
$Y\in B_1$, $Z\in B_{1b}$ there exists $T_0$ such that the forward-backward system
\eqref{eqabstrforbackfr} has unique mild solution $(b(t),f(t))$ for all $T<T_0$,
which depends locally Lipschitz on $Y$ and $Z$.
\end{theorem}

\section{Dual Banach triples and related nonlinear equations}
\label{secdualtriples}

In applications the two Banach triples participating in (AFB) are usually linked by duality,
and operators $A, A^b$ are dual, which allows one to derive the required properties of $A^b$
from the corresponding properties of $A$.

It is convenient to introduce the corresponding setting in an abstract way.
Following \cite{Kobook19}, we shall call a pair of Banach space $(B,C)$ a {\it dual pair},
if each of these spaces is a closed subset of the dual of the respective other space that separates
the points of the latter. Of course, if $B'$ is the Banach dual of $B$, then both $(B,B')$
and $(B',B)$ are dual pairs. But our symmetric notion includes also the pairs like $(C_{\infty}(\R^d), L_1(\R^d))$,
which are crucial for our present discussion. The duality is a bilinear form $(B,C)\to \R$, which
will be always denoted by $(b,c)$, $b\in B, c\in C$. Recall that the requirement that $B$ is a subset of the dual to $C$
means the following formula for the norm:
\[
\|b\|_B=\sup_{\|c\|\le 1} |(b,c)|.
\]

Let us say that the Banach triple $B_2\subset B_1\subset B$ is {\it generated} by the vector-valued operator
$D=(D_1, \cdots, D_n)$ if
\begin{equation}
\label{eqnormforgentri}
\|b\|_1=\|b\|+\sum_{j=1}^n \|D_j b\|,
\quad
\|b\|_2=\|b\|_1+\sum_{i\le j} \|D_iD_j b\|.
\end{equation}
Let the two Banach triples $B_2\subset B_1\subset B$ and $B_{2b}\subset B_{1b}\subset B_b$ (with the norms
of the second denoted $\|.\|_{2b}$, $\|.\|_{1b}$ and $\|.\|_b$) be generated by the operators
$D=(D_1, \cdots, D_n)$ and $D^b=(D^b_1, \cdots, D^b_n)$ respectively.
Let us say that the triples are {\it dual} if $(B,B_b)$ forms a dual pair and operators $D_j$ and $-D^b_j$ are dual,
in the sense that
 \[
 (D_jb,f)=-(b,D^b_jf), \quad b\in B_1, \, f\in B_{1b}, \, j=1, \cdots, n.
 \]
With some abuse of notation, we shall write just $D$ also for the operator $D^b$, so that
the previous equation becomes $(D_jb,f)=-(b, D_jf)$, and we shall say that the dual triples
are generated by $D$. Of course in all applications, both $D$ and $D^b$ are the gradients,
so that this unified notation is fully natural.

When discussing duality it is useful to work with two versions of the smoothing property
\eqref{eqsmoothimngas}. Namely, let the Banach triple $B_2\subset B_1\subset B$ be
{\it generated} by $D=(D_1, \cdots, D_n)$ and $A$ generate a strongly continuous
semigroup $e^{At}$ in $B$. Let us say that the semigroup $e^{At}$ is {\it smoothing
from the right in the 1st order}, if \eqref{eqsmoothimngas} holds.
Since the triple is generated by $D$, this is equivalent to the estimate
\begin{equation}
\label{eqsmoothimngasgentr}
\|D_je^{At}b\|\le \ka t^{-\om} \|b\|, \quad j=1, \cdots , n, \quad t\in (0,1],
\end{equation}
for all $b\in B$ (possibly with another constant $\ka$ than in \eqref{eqsmoothimngas},
but with the same $\om$). 
Let us say that the semigroup $e^{At}$ is {\it smoothing from
 the left in the 1st order}, if the operators $e^{At}D_j$ defined on $B_1$ can be extended
to bounded operators in $B$ for any $t>0$ such that
\begin{equation}
\label{eqsmoothimngasgentrleft}
\|e^{At}D_jb\|\le \ka t^{-\om} \|b\|, \quad j=1, \cdots , n, \quad t\in (0,1],
\end{equation}
for all $b\in B$.

For a Banach triple
generated by $D$ let us say that the semigroup $e^{At}$ is {\it smoothing from the right in the 2nd order}, if
\begin{equation}
\label{eqsmoothimngasgentr2nd}
\|D_je^{At}b\|_1\le \ka t^{-\om} \|b\|_1, \quad j=1, \cdots , n, \quad t\in (0,1],
\end{equation}
holds, and
it is {\it smoothing from the left in the 2nd order}, if
\begin{equation}
\label{eqsmoothimngasgentl2nd}
\|e^{At}D_jb\|_1\le \ka t^{-\om} \|b\|_1, \quad j=1, \cdots , n, \quad t\in (0,1],
\end{equation}
holds.

The next result supplies some links between various smoothing properties.
\begin{theorem}
\label{thAindualtriples}
Let the two Banach triples $B_2\subset B_1\subset B$ and $B_{2b}\subset B_{1b}\subset B_b$,
generated by the operators $D=(D_1, \cdots, D_n)$, be dual.
Let the operators 
\[
A\in \LC(B_2,B), \quad A^b\in \LC(B_{2b},B_b)
\]
be dual, $A^b=A^*$, in the sense that
 \[
 (Ab,f)=(b,A^*f), \quad b\in B_2, \, f\in B_{2b},
 \]
and let $A$ and $A^b=A^*$ generate strongly continuous semigroups $e^{At}$ in $B$ and $e^{A^*t}$ in $B_b$, having cores
 contained in $B_1$ and $B_{1b}$ respectively.

(i) If $e^{At}$ is smoothing in the 1st order from the right (respectively, from the left),
then $e^{A^*t}$ is smoothing in the 1st order from the left (respectively, from the right),
with the same parameter $\om$, and vice versa.

(ii) If the commutators $[D_j, e^{At}]$ extend to uniformly bounded (in $t\in [0,T]$ for any $T$)
operators in $B$, then (a) $e^{At}$ is a bounded semigroup in $B_1$ (that is,
the second estimate of \eqref{AgeninBB1} follows from the 1st one) and  (b) $e^{At}$
is smoothing in the 1st order from the right if and only if
it is smoothing in the 1st order from the left.

(iii)  $[D_j, e^{At}]$ extend to uniformly bounded operators in $B$ if and only if
$[D_j, e^{A^*t}]$ extend to uniformly bounded operators in $B_b$.

(iv) If $e^{A^*t}$ is smoothing from the left of the 1st order, then $e^{At}$ is smoothing
from the left in the 2nd order.

(v) If the commutators $[D_j, e^{At}]$ extend to uniformly bounded (in $t\in [0,T]$ for any $T$)
operators in $B_1$, then (a) $e^{At}$ is a bounded semigroup in $B_2$ and (b) $e^{At}$
is smoothing in the 2nd order from the right if and only if
it is smoothing in the 2nd order from the left.
\end{theorem}

\begin{proof}
(i) Let $e^{At}$ be smoothing in the 1st order from the right.
If $f\in B_{1b}$ belongs to the core of $e^{A^*t}$, then
\[
\|D_j^b e^{A^*t}f\|_b=\sup_{\|b\|\le 1} |(b,D_j^b e^{A^*t}f)|
=\sup_{\|b\|\le 1, b\in B_1} |(b,D_j^b e^{A^*t}f)|
= \sup_{\|b\|\le 1, b\in B_1} |(D_jb,e^{A^*t}f)
\]
\[
=\sup_{\|b\|\le 1, b\in B_1} |(e^{At}D_jb,f)|\le \ka e^{-\om t} \|f\|_b.
\]
Since the core is dense, the operator $D_j e^{A^*t}$ extends to the bounded operator $B_b\to B_b$.
Similarly, other statements are proved.

(ii) (a) If $b\in B_1$, then for $t\in [0,T]$,
\[
\|D_j e^{At}b\|\le \|e^{At} D_j b\|+\|[D_j,e^{At}]b\|\le M e^{mt} \|b\|_1 +C(T) \|b\|.
\]
and thus
\[
\|e^{At}b\|_1 \le  nM e^{mt} \|b\|_1+(1+C(T))\|b\|\le (nM e^{mt}+1+C(T)) \|b\|_1,
\]
so that the operators $e^{At}$ are bounded in $B_1$ uniformly for $t\in [0,T]$.
Statement (b) is straightforward.

(iii) This follows from the duality relation:
\[
([D_j, e^{At}]b, f)=(b,[D_j, e^{A^*t}]f).
\]

(iv) To estimate $\|e^{At}D_jb\|_1$ we need to estimate $\|D_ie^{At}D_jb\|$. We have
\[
\|D_ie^{At}D_jb\|=\sup_{f\in B_{1b},\|f\|_b\le 1}|(D_ie^{At}D_jb,f)|
=\sup_{f\in B_{1b},\|f\|_b\le 1}|(D_jb, e^{A^*t}D_if)|
\]
\[
\le \|b\|_1 \sup_{\|f\|_b\le 1}\|e^{A^*t}D_i f\|_b
\le \ka t^{-\om} \|b\|_1.
\]

(v) (a) In order to estimate $\|e^{At}b\|_2$, we need the estimate for $\|D_iD_je^{At}b\|$.
We have
\[
\|D_iD_je^{At}b\|=\|D_i e^{At} D_jb\|+\|D_i[e^{At},D_j] b\|
\le \|e^{At} D_iD_jb\|+\|[D_i, e^{At}] D_jb\|+\|D_i[e^{At},D_j] b\|
\]
\[
 \le \|e^{At}\|_{B\to B} \|b\|_2+\|[D_i, e^{At}]\|_{B_1\to B_1} \|b\|_2
 +\|[D_i, e^{At}]\|_{B_1\to B_1} \|b\|_1
\]
\[
 \le \left(\|e^{At}\|_{B\to B}+2\|[D_i, e^{At}]\|_{B_1\to B_1}\right) \|b\|_2.
 \]

(v) (b) Assume that $e^{At}$ is smoothing in the 2nd order from the left. To show that
it is smoothing in the 2nd order from the right, we need to
estimate $\|D_je^{At}b\|_1$ and thus $\|D_iD_je^{At}b\|$. We have
\[
\|D_iD_je^{At}b\|\le \|D_ie^{At}D_jb\|+\|D_i[D_j,e^{At}]b\|,
\]
and both terms are estimated by $\|b\|_1$, as required.
\end{proof}

For dual Banach triples and dual generators, we can reformulate Theorems \ref{thabsrforback}
and \ref{thabsrforbackfr} almost without any additional assumptions on $A^*$.
The following result is a direct consequence of Theorems
 \ref{thabsrforback}, \ref{thabsrforbackfr} and \ref{thAindualtriples}.
  
\begin{theorem}
\label{thforbackindual}
Let the two Banach triples $B_2\subset B_1\subset B$ and $B_{2b}\subset B_{1b}\subset B_b$,
generated by the operators $D=(D_1, \cdots, D_n)$, be dual.
Let the operators
\[
A\in \LC(B_2,B), \quad A^b\in \LC(B_{2b},B_b)
\]
be dual, $A^b=A^*$, in the sense that
 \[
 (Ab,f)=(b,A^*f), \quad b\in B_2, \, f\in B_{2b},
 \]
and let $A$ and $A^b=A^*$ generate strongly continuous semigroups $e^{At}$ in $B$ and $e^{A^*t}$ in $B_b$, having cores
contained in $B_1$ and $B_{1b}$ respectively. Moreover, $A$ is smoothing of the 1st order from the right or from the left
and the operators $[e^{At},D_j]$ extend to uniformly bounded operators $B\to B$.
Let the assumptions (iii) and (iv) about $H, H^b$, $u$ of conditions (AFB) hold.
Then, for any $Y\in B_1$, $Z\in B_{1b}$ there exists $T_0$ such that the forward-backward systems
\eqref{eqabstrforback} and
\eqref{eqabstrforbackfr} with $A^b=A^*$ have unique mild solutions for all $T<T_0$.
\end{theorem}

Let us specify the abstract setting to a more concrete forward-backward system consisting of the coupled McKean Vlasov (forward) and HJB (backward) equations for $t\in [a,T]$:

\begin{equation}
\label{eqforbackHJBMVfr}
\left\{
\begin{aligned}
& D_{a+*} g(t,x)=A^*g(t,x)+\sum_{j=1}^d h_j(t,g(t,.),u(t,x,\frac{\pa f}{\pa x}(t,x))) \frac{\pa g}{\pa x_j},
\quad g(a)=Y, \\
& D_{T-*} f(t,x)=-A f(t,x)+H(t,x,\frac{\pa f}{\pa x}(t,x), g_{\ge t}), \quad f(T)=Z.
\end{aligned}
\right.
\end{equation}

Here the dual Banach triples are $W_2(\R^d)\subset W_1(\R^d)\subset L_1(\R^d)$ for the first (forward) equation
and $C^2_{\infty}(\R^d)\subset C^1_{\infty}(\R^d)\subset C_{\infty}(\R^d)$ for the second (backward) equation,
both generated by the derivative operator $D=\pa/\pa x$. The corresponding norms were defined at
the end of Section \ref{abstrMVHJB}.
The Banach space $B^{con}$ from assumption (AFB)
is the space of $U$-valued continuous functions on $R^d$, $U$ a compact subset of Euclidean space.
At the end of Section \ref{abstrMVHJB} were given examples of the
generators $A$ that satisfy the requirements of Theorem \ref{thforbackindual}, which supplies the conditions for
local well-posedness for system \eqref{eqforbackHJBMVfr}.

\section{Fractional McKean-Vlasov and HJB equations on manifolds}
\label{seceqman} %secdualtriples

Let
  \begin{equation}
\label{Lapgeom}
\De_{LB} \phi ={div} \, (\nabla \phi) =\frac{1}{\sqrt {\det g}}
\sum_{j,k}\frac{\pa}{\pa x_j} \left(\sqrt {\det g} \, g^{jk}\frac{\pa}{\pa x_k}\right)
 \end{equation}
denote the Laplace-Beltrami operator on a compact Riemannian manifold $(M,g)$ of dimension $N$,
with the Riemannian metric given by the matrix $g=(g_{jk}(x))$ and its inverse matrix $G=(g^{jk}(x))$.
Let $K(t,x,y)$ be the corresponding heat kernel, that is, $K(t,x,y)$ is the solution of the heat equation
$(\pa K/\pa t)=\De_{LB} K$ as a function of $(t>0,x\in M)$ and has the Dirac initial condition $K(0,x,y)=\de_y(x)$.
It is well known that the Cauchy problem for this heat equation is well posed in $M$ and the resolving operators
\begin{equation}
\label{eqheatsem}
S_tf(x)=e^{t\De_{LB}}=\int_M K(t,x,y) f(y) \, dv(y),
\end{equation}
where $dv(y)$ is the Remannian volume on $M$, form a strongly continuous semigroup of contractions
(the Markovian semigroup of the Brownian motion in $M$) in the space $C(M)$
of bounded continuous functions on $M$, equipped with the sup-norm.

Let $C^1(M)$ denote the space of
continuously differentiable functions on $M$ equipped with the norm
\begin{equation}
\label{eqnormC1M}
\|f\|_{C^1(M)}=\|f\|_{C(M)}+\sup_x \|\nabla f(x)\|_x,
\end{equation}
where in local coordinates
\begin{equation}
\label{eqmagngrad}
\|\nabla f(x)\|_x^2=\left(\nabla f(x), G(x)\nabla f(x)\right)
=\sum_{jk} g^{jk}(x) \frac{\pa f}{\pa x_j} \frac{\pa f}{\pa x_k},
\end{equation}
and $C^2(M)$ denote the space of twice continuously differentiable functions on $M$ equipped with the norm
\begin{equation}
\label{eqnormC2M}
\|f\|_{C^2(M)}=\|f\|_{C^1(M)}+\sup_x \|\nabla^2 f(x)\|_x,
\end{equation}
where in local coordinates
\begin{equation}
\label{eqmagnsecgrad}
\|\nabla ^2f(x)\|_x^2
=\sum_{jk} g^{jk}(x) g^{im}(x) \frac{\pa ^2f}{\pa x_j \pa x_i} \frac{\pa ^2f}{\pa x_k \pa x_m}.
\end{equation}

%\begin{remark}
The gradient $\nabla f(x)$ is an element of $T^*M_x$, the cotangent space to $M$ at $x$,
and $\nabla^2 f$ is a tensor of type $(0,2)$. Formulas \eqref{eqmagngrad} and \eqref{eqmagnsecgrad}
represent the standard lifting of the Riemannian metric to tensors.
%\end{remark}

The dual Banach triples used for the analysis of equations on manifolds
are the natural analogs of the triples used for equations in $\R^d$. Namely,
these are the triple $C_2(M)\subset C_1(M)\subset C(M)$ (with the norms introduced above)
and the triple of Sobolev functional spaces  $W_2(M)\subset W_1(M)\subset L_1(M)$,
with the norms
\[
 \|f\|_{L_1(M)}=\int_M |f(x)| \, dv(x), \quad \|f\|_{W_1(M)}=\|f\|_{L_1(M)}+\int_M  \|\nabla f(x)\|_x  \, dv(x),
\]
\[
\|f\|_{W_2(M)}=\|f\|_{W_1(M)}+\int_M  \|\nabla^2 f(x)\|_x  \, dv(x).
\]

It is known that the semigroup $S(t)$ in $C(M)$ has $C^2(M)$
as its invariant core. It is also strongly continuous as  a semigroup in $L_1(M)$
(actually in all $L_p(M)$, $p\ge 1$), see \cite{AponLevyonman} and references therein.

\begin{remark}
\label{reminvartriple}
The identification of $\nabla f$ with the collection of $n$ functions
$\pa f/\pa x_j$, which is required to comply with the setting of Sections \ref{abstrMVHJB}
and \ref{secdualtriples},
can be justified only in local coordinates, but not globally. In order to have invariant
theory we cannot assume that the operator $D=(D_1,\cdots, D_n)$ act from $B=C(M)$ to $B^n$.
In the invariant exposition the gradient $D=\nabla$ maps $C(M)$ to the space $FT^*M$ of
covector fields on $M$, and $D^2=\nabla^2$ maps $C(M)$ to the symmetric tensor fields
of type $(0,2)$. However, the topologies of $C^1(M)$ and $C^2(M)$ (as well as $W_1(M)$
and $W_2(M)$) remain to be generated
by $D=\nabla$ and $D^2$ respectively, according to \eqref{eqnormC1M}, \eqref{eqnormC2M},
in analogy with \eqref{eqnormforgentri}. Thus the whole theory is recovered in this
slightly generalized setting.
\end{remark}

The key smoothing and smoothness preservation properties of this semigroup
needed for our theory are collected in the next result.

\begin{theorem}
 \label{propheatsemomansmoo}
 (i) The operators $S_t$ are smoothing from the right and from the left:
\begin{equation}
\label{eq1propheatsemomansmoo}
\|\nabla S_tf\|_{C(M)}\le C t^{-1/2} \|f\|_{C(M)}, \quad \|S_t \nabla f\|_{C(M)}\le C t^{-1/2} \|f\|_{C(M)},
 \end{equation}
with a constant $C$, uniformly for any compact interval of time.

(ii) The operators $S_t$ are smoothness preserving:
\begin{equation}
\label{eq2propheatsemomansmoo}
\|S_tf\|_{C^1(M)}\le C \|f\|_{C^1(M)},
 \end{equation}
with a constant $C$, uniformly for any compact interval of time.

 (iii) The commutators $[\nabla,S_t]$ extend to bounded operators in $C(M)$,
 uniformly for any compact interval of time.
 \end{theorem}

\begin{proof} The first estimate in \eqref{eq1propheatsemomansmoo}
 is a consequence of the well-known estimate for the derivatives of the heat kernel
on a compact Riemannian manifold (see Theorem 6 in \cite{Davies89}):
\begin{equation}
\label{eqheatsemder}
\|\nabla K(t,x,y)\|_M \le C(\de,N) t^{-N/2} t^{-1/2} \exp \left\{ -\frac{d^2(x,y)}{(4+\de)t}\right\},
\end{equation}
with any $\de>0$ and a constant $C(\de,N)$, where the derivative $\nabla$
is taken with respect to $x$, and where $d$ is the Riemannian distance in $M$.
In fact, differentiating \eqref{eqheatsem} and using \eqref{eqheatsemder} yields
the first estimate \eqref{eq1propheatsemomansmoo}.

As shown in Theorem \ref{thAindualtriples}, the second estimate  in \eqref{eq1propheatsemomansmoo} follows from
the first estimate and (iii). Also (ii) follows from (iii). Thus it remains to show (iii).
And this follows from \eqref{eqheatsemder} and the method of parametric (frozen coefficients) approximation.
This method (see e.g. formula (5.60) in \cite{Kobook19}) starts by representing $K$ in terms of its asymptotics $K_{as}$
and the integral correction as
\begin{equation}
\label{eq3propheatsemomansmoo}
K(t,x,y)= K_{as}(t,x,y)+\int_0^t K(t-s,x,z)F(s,z,y) \, ds,
\end{equation}
where $F$ is the error term in the equation for $K_{as}$, that is
\[
\frac{\pa K_{as}}{\pa t}(t,x,y)-\De_{LB} K_{as}(t,x,y)=-F(t,x,y).
\]
From \eqref{eqheatsemder} it follows that the derivative of the second term in \eqref{eq3propheatsemomansmoo}
is bounded and thus the statement about commutator reduces to the statement that the integral operator with the
integral kernel
\[
\frac{\pa K_{as}}{\pa x}(t,x,y)-\frac{\pa K_{as}}{\pa y}(t,x,y)
\]
is bounded in $C(M)$. And this is performed as in the case of heat equations in $\R^d$, as
it is implied by the fact that $K_{as}$ is Gaussian with 'frozen coefficients'.
\end{proof}

The following result is the direct consequence of Theorems \ref{propheatsemomansmoo}
and  \ref{thAindualtriples}, and the well known fact that the operator $\De_{LB}$ is
self-dual with respect to the coupling given by the integration with respect to the
volume measure $dv$ on $M$.

\begin{theorem}
 \label{propheatsemomansmoodu}
 The smoothing properties of the operators $S(t)$ also hold in the integral
 norms, that is,
\begin{equation}
\label{eq1propheatsemomansmoodu}
\|\nabla S_tf\|_{L_1(M)}\le C t^{-1/2} \|f\|_{L_1(M)},
\quad \|S_t \nabla f\|_{L_1(M)}\le C t^{-1/2} \|f\|_{L_1(M)},
 \end{equation}
 \begin{equation}
\label{eq2propheatsemomansmoodu}
\|S_tf\|_{W_1(M)}\le C \|f\|_{W_1(M)},
 \end{equation}
and the commutators $[\nabla,S_t]$ extend to bounded operators in $L_1(M)$,
 uniformly for any compact interval of time.
 \end{theorem}

For the stochastic control of diffusions on $(M,g)$ with the second order part being fixed as $\De_{LB}$,
and where control is carried out via the drift only (drift control of the Brownian motion on $M$),
the corresponding HJB equation is the equation
\begin{equation}
\label{eqHJBmanif}
\frac{\pa f}{\pa t}(t,x)= \De_{LB} f(t,x) +H(x,\nabla f (t,x)),
\end{equation}
where the Hamiltonian $H(x,p)$, $x\in M, p\in T^*M_x$ is a function on the cotangent bundle $T^*M$ of the form
\begin{equation}
\label{eqHJBmanif1}
H(x,p)=\sup_{u\in U} [(g(x,u),p)+J(x,u)],
\end{equation}
where $U$ is a compact set of possible controls and $J,g$ are some continuous functions
 and a vector field depending on $u$ as a parameter.
In case of zero-sum stochastic two-player games with the so-called Isaac's condition,
the Hamiltonian function takes the form
\begin{equation}
\label{eqHJBmanif2}
H(x,p)=\sup_{u\in U}\inf_{v\in V} [(g(x,u,v),p)+J(x,u,v)]=\inf_{v\in V}\sup_{u\in U} [(g(x,u,v),p)+J(x,u,v)].
\end{equation}
The possibility to exchange $\sup$ and $\inf$ here is called Isaac's condition. It is fulfilled, in particular,
when the control of two players can be separated in the sense that the Hamiltonian becomes
\begin{equation}
\label{eqHJBmanif3}
H(x,p)=\sup_{u\in U} [(g_1(x,u),p)+J_1(x,u)]+\inf_{v\in V}[(g_2(x,v),p)+J_2(x,v)]+J_0(x).
\end{equation}

The fractional version of equation \eqref{eqHJBmanif} is the fractional HJB equation
\begin{equation}
\label{eqHJBmaniffr}
D_{a+*}f(t,x)= \De_{LB} f(t,x) +H(x,\nabla f (t,x))
\end{equation}
(fractional derivative acts on the variable $t$ and $\nabla$ and $\De_{LB}$ act on the variable $x\in M$),
which describes the cost functions for the scaled continuous-time random walks with the spatial
motion govern by $\De_{LB}$ (see \cite{KoVe14}).

The McKean-Vlasov equation on $M$ describing a nonlinear diffusion on $M$ (with the operator $\De_{LB}$ as the main part and a nonlinear drift $h$) is the equation of the form
\begin{equation}
\label{eqMVmanif}
\frac{\pa f}{\pa t}(t,x)= \De_{LB} f(t,x) +(h(x, f(t,.)),\nabla f (t,x)),
\end{equation}
where $h$ maps pairs $(x, f(t,.))$ to the elements of the tangent space $T_xM$ of $M$ at $x$
(so that the coupling $(h(x, f(t,.)),\nabla f (t,x))$ yields a well defined function on $M$).
Thus $h$ can be looked at as a vector field on $M$ depending on $f$ as a parameter. The
dependence of $h$ on $f$ via certain integrals, that is,
\begin{equation}
\label{eqMVmanifHam}
h(x, f(t,.))=h(x, \int_M g(x) f(t,x)  \, dv(y)),
\end{equation}
where $g$ is some bounded function on $M$, is typical in applications.

The fractional version of equation \eqref{eqMVmanif} is, of course, the equation
\begin{equation}
\label{eqMVmaniffr}
D_{a+*}f(t,x)= \De_{LB} f(t,x) +(h(x, f(t,.)),\nabla f (t,x)).
\end{equation}

The following results are direct consequences of Theorems
\ref{propheatsemomansmoo}, \ref{propheatsemomansmoodu}, \ref{thabsrtHJBMVwell}
and \ref{thabsrtHJBMVwellfr} (and Remark \ref{reminvartriple}).

\begin{theorem}
\label{thHJBwellma}
(i) Let $H(x,p)$ be a continuous function on the cotangent bundle
$T^*M$ to the compact Riemannian manifold $(M,g)$ such that
\begin{equation}
\label{eq1thHJBwellma}
|H(x,p_1)-H(x,p_2)| \le L_H \|p_1-p_2\|_x
\end{equation}
for all $x$ with a constant $L_H$. Then for any $Y\in C^1(M)$ there exists
 a unique mild solution for the Cauchy problem
of equation \eqref{eqHJBmanif} and a unique solution for the Cauchy problem of equation
\eqref{eqHJBmaniffr} with initial condition $f(0,.)=Y$. These solutions depend Lipschitz
continuously on the initial data in the norm of $C^1(M)$.

(ii) Let $h(x, f(.))$ be a uniformly bounded vector field on $M$ depending Lipschitz continuously on $f$
in the norm of $L_1(M)$:
\begin{equation}
\label{eq2thHJBwellma}
\|h(x,f_1(t,.))-h(x,f_2(t,.))\|_x \le L_h \|f_1(t,.)-f_2(t,.)\|_{L_1(M)},
\end{equation}
for all $x\in M$ with a constant $L_h$.
Then for any $Y\in W_1(M)$ there exists a unique mild solution for the Cauchy problem
of equation \eqref{eqMVmanif} and a unique solution for the Cauchy problem of equation
\eqref{eqMVmaniffr} with initial condition $f(0,.)=Y$. These solutions depend Lipschitz
continuously on the initial data in the norm of $W_1(M)$.
\end{theorem}

\section{Fractional forward-backward systems on manifolds}
\label{seceqsysman}

Finally, we analyse the analogue of system \eqref{eqforbackHJBMVfr} on manifolds with the
generator $A=\De_{LB}$:
\begin{equation}
\label{eqforbackHJBMVfrma}
\left\{
\begin{aligned}
& D_{a+*} g(t,x)=\De_{LB}g(t,x)+(h(t,x, g(t,.),u(t,x, \nabla f(t,x))) \nabla g (t,x),
\quad g(a)=Y, \\
& D_{T-*} f(t,x)=-\De_{LB} f(t,x)+H(t,x,\nabla f(t,x), g_{\ge t}), \quad f(T)=Z.
\end{aligned}
\right.
\end{equation}

The following result is a consequence of Theorems \ref{thforbackindual},
 \ref{propheatsemomansmoo} and \ref{propheatsemomansmoodu}.

\begin{theorem}
\label{thforbackma}
Let $U$ be a compact set in a Euclidean space and
$u(t,x,p)\in U$ be a continuous function of the triple $t\in \R$, $x\in M$, $p\in T^*M_x$
which is Lipschitz continuous in the last argument:
\begin{equation}
\label{eqcontLipconma}
|u(t,x, p)-u(t,x, \tilde p)|
\le L_u \|p-\tilde p\|_x.
\end{equation}

Let $h:\R\times M\times L_1(M)\times U\to TM$ be a continuous
mapping such that $h(t,x,f,u)\in TM_x$ for all $x,f,u$,  which
is Lipschitz in the sense that
\begin{equation}
\label{HamLipschrepma}
\|h(t,x,g(t,.), u)-h(t,x,\tilde g(t,.),\tilde u)\|_x
\le L_h \|g(t,.)-\tilde g(t,.)\|_{L_1(M)}+L_{hu} |u-\tilde u|,
\end{equation}

Let $H(t,x,p, g_{\ge t})$, $t\in \R, x\in M, p\in T^*M_x, g\in L_1(M)$
be a continuous function,
which is Lipschitz in the sense that
\begin{equation}
\label{HamLipschrepbma}
|H(t,x,p,g(.))-H(t,x, \tilde p,g(.))|
\le L_H \|p-\tilde p\|_x ,
\end{equation}
\begin{equation}
\label{HamLipschparrepbma}
|H(t,x,p,g(.)) -H(t,x,p,\tilde g(.))|
\le L_H\|f-\tilde f\|_{C([a,T],B)} (1+\sum_{j=0}^n \|f_j\|_b).
\end{equation}
Then for all $Y,Z$  there exist $T_0$ such that the system \eqref{eqforbackHJBMVfrma}
has a unique solution for all $T\le T_0$.
\end{theorem}

\end{document}